\documentclass[10pt]{amsart} 
\usepackage{amsmath,amsthm, amssymb, amsfonts, epsfig,amscd}
\usepackage[all]{xy}

\usepackage[mathscr]{eucal}
\usepackage{mathrsfs}
\usepackage{latexsym}
\usepackage{graphicx}
\usepackage[all]{xy}
\usepackage[table]{xcolor}
\usepackage[latin1]{inputenc}
\usepackage{tikz}
\usetikzlibrary{trees}

\theoremstyle{plain}
\newtheorem{thm}[subsection]{Theorem}
\newtheorem{lem}[subsection]{Lemma}
\newtheorem{prop}[subsection]{Proposition}
\newtheorem{cor}[subsection]{Corollary}
\newtheorem{lemma}[subsection]{Lemma}
\newtheorem{theorem}[subsection]{Theorem}

\theoremstyle{definition}
\newtheorem{rem}[subsection]{Remark}

\newtheorem{remark}[subsection]{Remark}

\def\Para{\ensuremath{\mathcal P}}

\def\vs{\vskip}
\def\ni{\noindent}

\begin{document}
\title[Cotangent bundle to the Flag variety-I]
{Cotangent bundle to the Flag variety-I}
\author[V. Lakshmibai]{V. Lakshmibai${}^{\dag}$}
\address{Northeastern University, Boston, USA}
   \email{lakshmibai@neu.edu}
\author[C.S. Seshadri]{C.S. Seshadri}
\address{Chennai Mathematical Institute, Chennai, India}
   \email{css@cmi.ac.in}
\author[R. Singh]{R. Singh}
\address{Northeastern University, Boston, USA}
   \email{singh.rah@husky.neu.edu}
\thanks{${}^{\dag}$ V.~Lakshmibai was partially supported by
NSA grant H98230-11-1-0197, NSF grant 0652386}

\maketitle

\begin{abstract}%
We show that there is a ${SL_n}$-stable closed subset of an
affine Schubert variety in the infinite dimensional Flag variety
(associated to the Kac-Moody group ${\widehat{SL_n}}$) which is a
natural compactification of the cotangent bundle to the
finite-dimensional Flag variety ${{SL_n/B}}$.
\end{abstract}

\section{Introduction} Let the base field $K$ be the field of
complex numbers. Consider a cyclic quiver with $h$ vertices and
dimension vector ${\underline d}=(d_1,\cdots, d_h)$:
$$\xymatrix{ 1 \ar[r] & 2 \ar[r] & \cdots & h-2 \ar[r] &
h-1\ar[dll]\\ &&h \ar[ull] & & }$$ Denote $V_i=K^{d_i}$. Let
$$Z=Hom(V_1,V_{2})\times \cdots\times Hom(V_h,V_{1}),\
GL_{{\underline d}}=\prod_{1\le i\le h}\,GL{(V_i)}$$ We have a
natural action of $GL_{{\underline d}}$ on $Z$: for
$g=(g_1,\cdots,g_h)\in GL_{{\underline d}},f=(f_1,\cdots,f_h)\in
Z$,
$$g\cdot f=(g_2f_1g_1^{-1},g_3f_2g_2^{-1},\cdots,g_1f_hg_h^{-1})$$
Let $$\mathcal{N}=\{(f_1,\cdots,f_h)\in Z\,|\,f_h\circ
f_{h-1}\circ\cdots\circ f_1:V_1\rightarrow V_1{\mathrm{\ is\
nilpotent}}\}$$
Note that $f_h\circ f_{h-1}\circ\cdots\circ f_1:V_1\rightarrow
V_1$ being nilpotent is equivalent to

\ni $f_{i-1}\circ f_{i-2}\circ\cdots\circ f_1\circ f_h
\circ\cdots\circ f_{i+1} f_i:V_i\rightarrow V_i$ being nilpotent.
Clearly $\mathcal{N}$ is $GL_{{\underline d}}$-stable. Lusztig
(cf.\cite{lus}) has shown that an orbit closure in $\mathcal{N}$
is canonically isomorphic to an open subset of a Schubert variety
in ${\widehat{SL}}_n/Q$, where $n=\sum_{1\le i\le h}\, d_i$, and
$Q$ is the parabolic subgroup of ${\widehat{SL}}_n$ corresponding
to omitting $\alpha_0,\alpha_{d_{1}},
\alpha_{d_{1}+d_{2}},\cdots,\alpha_{d_{1}+\cdots +d_{h-1}}$
($\alpha_i, 0\le i\le n-1$ being the set of simple roots for
${\widehat{SL}}_n$). Corresponding to $h=1$, we have that
$\mathcal{N}$ is in fact the variety of nilpotent elements in
$M_{d_{1},d_{1}}(K)$, and thus the above isomorphism identifies
$\mathcal{N}$ with an open subset of a Schubert variety
$X_{\mathcal{N}}$ in ${\widehat{SL}}_n/G_0$, $G_0$
being the maximal parabolic subgroup of ${\widehat{SL}}_n$
corresponding to ``omitting" $\alpha_0$.

Let now $h=2$
$$Z_0=\{(f_1,f_2)\in Z\,|\,f_2\circ f_1=0,f_1\circ f_2=0\}$$
Strickland (cf. \cite{st}) has shown that each irreducible
component of $Z_0$ is the conormal variety to a determinantal
variety in $M_{d_1,d_2}(K)$. A determinantal variety in $M_{d_1,d_2}(K)$
being canonically isomorphic to an open subset in a certain
Schubert variety in $G_{d_2,d_1+d_2}$ (the Grassmannian variety of
$d_2$-dimensional subspaces of $K^{d_1+d_2}$) (cf.\cite{g/p-2}),
the above two results of Lusztig and Strickland suggest a
connection between conormal varieties to Schubert varieties in the
(finite-dimensional) flag variety and the affine Schubert
varieties. This is the motivation for this article. 
Let $G=SL_n$.

Inspired by Lusztig's  embedding of $\mathcal N$ in $\widehat {SL_n}/Q$,
we define a family of maps $\psi_p:T^*G/B\hookrightarrow\widehat{SL}_n/\mathcal{B}$, parametrized by polynomials in one variable  with coefficients in $\mathbb C((t))$, and with $1$ as the constant term.
For a particular map $\phi$ (analogous to Lusztig's map) in this family, we find a $\kappa_0\in\widehat W$ such that the affine Schubert variety $X(\kappa_0)$ is $G_0$-stable ($G_0$ being as above, the
maximal parabolic subgroup of ${\widehat{SL}}_n$ corresponding to
``omitting" $\alpha_0$) and show that $\phi$ gives an embedding $T^*G/B\hookrightarrow X(\kappa_0)\subset{\widehat{SL}}_n/\mathcal{B}$.
We thus obtain a ${{SL_n}}$-stable closed subvariety of $X(\kappa_0)$ as a natural
compactification of $T^*G/B$ (cf. Theorem \ref{ctgt}). Let
$\pi:{\widehat{SL}}_n/\mathcal{B}\rightarrow
{\widehat{SL}}_n/G_0$  be the canonical projection. Then we show
that $\pi(T^*G/B)=\mathcal{N}$, the variety of nilpotent
matrices, and that $\pi\,|_{T^*G/B}:T^*G/B\rightarrow \mathcal{N}$
is in fact the Springer resolution.

Following the above ideas, Lakshmibai (cf. \cite{ls}) has obtained a stronger result for $T^*G_{d,n}$, the cotangent bundle to the Grassmannian variety $G_{d,n}$. She shows that there is an embedding $\chi$ (analogous to $\phi$) of $T^*G_{d,n}$ inside a Schubert variety $X(\iota)\subset\widehat {SL_n}/\mathcal Q_d$ (where $\mathcal Q_d$ is the two step parabolic subgroup of $\widehat {SL_n}$ corresponding to omitting $\alpha_0,\alpha_d$) such that $X(\iota)$ \textbf{is in fact a compactification of}  $T^*G_{d,n}$.
The result of \cite{ls} has been generalized to $T^*G/P$ in \cite{comin}, $G/P$ being a cominuscule Grasssmannian variety.

It would be interesting to know if the result of \cite{ls} could be achieved replacing $P$ with $B$, for a suitable generalization of $\chi$.
We show in \S\ref{counter} that this is not possible for any $\psi_p$ in the above family, even when $n=3$. 
We think that our result about the embedding $\phi:T^*G/B\hookrightarrow X(\kappa_0)$ identifying a certain ${{SL_n}}$-stable closed subvariety of $X(\kappa_0)$ as a natural
compactification of $T^*G/B$ is the best possible in relating $T^*G/B$ and affine Schubert varieties in $\widehat{SL_n}/\mathcal B$.

The results of this paper open up other related problems like,  the study of 
line bundles on $T^*G/B, G=SL_n$ (using the embedding of $T^*G/B$ into $X(\kappa_0)$, and realizing line bundles on $T^*G/B$ as restrictions of suitable line bundles on $X(\kappa_0)$), establishing similar embeddings of the cotangent bundles to partial flag varieties $G/Q$ ($G$ semi-simple and $Q$ a parabolic subgroup) etc. Further, the facts that conormal varieties to Schubert varieties in $G/B$ are closed subvarieties of  $T^*G/B$, and that the affine Schubert variety $X(\kappa_0)$ contains a $G$-stable closed subvariety which is a natural
compactification of $T^*G/B$, suggest similar compactifications
for conormal varieties to Schubert varieties in $G/B$ (by
suitable affine Schubert varieties in
${\widehat{SL}}_n/\mathcal{B})$; such a realization could lead to important consequences such as a knowledge of  the equations of the conormal varieties (to Schubert varieties) as subvarieties of  the cotangent bundle. These problems will be dealt with in a subsequent paper.

Regarding results on similar compactifications, we mention Mirkovic-Vybornov's work (cf. \cite{m-v}), where the authors construct compactifications of Nakajima's quiver varieties of type $\mathbf A$ inside affine Grassmannians of type $\mathbf A$. Also, Manivel and Michalek (\cite{m-m}) have recently studied the local geometry of tangential varieties (which are compactifications of the tangent bundle) to cominuscule Grassmannians.

The sections are organized as follows. In \S \ref{prelim}, we fix
notation and recall \emph{affine Schubert varieties}. In \S
\ref{elt}, we introduce the elements $\kappa$ and $\kappa_0$ (in ${\widehat{W}}$,
the affine Weyl group), and prove some properties of $\kappa$. In
\S \ref{main}, we prove a crucial result on $\kappa$ needed for
realizing the embeddings of $\mathcal{N}$ and $T^*SL_n/B$ inside
${\widehat{SL}}_n/G_0$ and ${\widehat{SL}}_n/\mathcal{B}$
respectively. In \S \ref{lumap}, we spell out Lusztig's
isomorphism which identifies $\mathcal{N}$ with an open subset of
$X_{G_0}(\kappa)$ (inside ${\widehat{SL}}_n/G_0$).
In \S \ref{bundle}, using the map $\phi:T^*G/B\rightarrow \widehat{SL}_n/\mathcal B$ as above and the natural projection $\widehat{SL}_n/\mathcal B\rightarrow \widehat{SL}_n/G_0$, we recover the Springer resolution of $\mathcal N$;
we also prove the main result that $\phi$ identifies an $SL_n$-stable closed subvariety of $X(\kappa_0)$ as a compactification of $T^*G/B$.
In \S \ref{counter}, we show that it is not possible, for any choice in the family $\psi_p$, to realize an affine Schubert variety (in $\widehat{SL}_3/\mathcal B$) as a compactification of the cotangent bundle $T^*SL_3/B$.

\vs.2cm\ni\textbf{Acknowledgement:} The first author wishes to
thank Chennai Mathematical Institute for the hospitality extended
to her during her visit in Winter 2014 when this work was started. 
The authors thank the Macaulay2 team for providing a 
web client to the same at http://habanero.math.cornell.edu:3690/
\section{Affine Schubert varieties}\label{prelim} Let $K=\mathbb{C}, F=K((t))$,
the field of Laurent series, $A=K[[t]]$. Let $G$ be a
semi-simple algebraic group over $K$, $T$ a maximal torus in $G$,
$B$ a Borel subgroup, $B\supset T$, and let $B^{-}$ be the Borel subgroup opposite to $B$. 
Let $\mathcal{G}=G(F)$. The natural inclusion
$K\hookrightarrow A\hookrightarrow F$ induces an inclusion
$$G\hookrightarrow G(A)\hookrightarrow\mathcal{G}$$ The
natural projection
$A\rightarrow K$ given by $t\mapsto 0$
induces a homomorphism
$$\pi:G(A)\rightarrow G$$ 
The group $\mathcal{B}:=\pi^{-1}(B)$ is a Borel subgroup of $\mathcal G$. 


\subsection{Bruhat decomposition:}  Let ${\widehat W}=N(K[t,t^{-1}])/T$, the \emph{affine Weyl group} of
$G$ (here, $N$ is the normalizer of $T$ in $G$). 
The group $\widehat W$ is a Coxeter group (cf. \cite{kac}).
We have that $$G(F)=\dot\cup_{w\in{\widehat W}}\mathcal{B}w\mathcal{B},
G(F)/\mathcal{B}=\dot\cup_{w\in{\widehat
W}}\mathcal{B}w\mathcal{B}(mod\,\mathcal{B})$$ For $w\in{\widehat
W}$, let $X(w)$ be the \emph{affine Schubert variety} in
$G(F)/\mathcal{B}$:
$$X(w)=\dot\cup_{\tau\le w}\mathcal{B}\tau\mathcal{B}(mod\,
\mathcal{B})$$ It is a projective variety of dimension $\ell(w)$.

\subsection{Affine Flag variety,  Affine Grassmannian:}

Let $G=SL(n)$, $\mathcal{G}=G(F)$, $G_0=G(A)$. 
We say $g\in\mathcal G$ is integral if and only if $g\in G_0$, i.e viewed as $G$-valued meromorphic function on $\mathbb C$, it has no poles at $t=0$.
The homogeneous space $\mathcal{G}/\mathcal{B}$ is the \emph{affine Flag variety}, and $\mathcal{G}/G_0$ is the \emph{affine Grassmannian}. 
Further, 
$$\mathcal{G}/G_0=\dot\cup_{w\in{\widehat W}^{G_0}}\mathcal{B}w\,G_0 (mod\,G_0)$$ 
where ${\widehat W}^{G_0}$ is the set of minimal representatives in ${\widehat W}$ of ${\widehat W}/W_{G_0}$.

\ni  Let $${\widehat
{Gr(n)}}=\{A\texttt{-lattices\ in\ } F^n \}$$ Here, by an
$A$-lattice in $F^n$, we mean a free $A$-submodule of $F^n$ of
rank $n$. Let $E$ be the standard lattice, namely, the $A$-span of
the standard $F$-basis $\{e_1,\cdots,e_n\}$ for $F^n$. For $V\in
{\widehat {Gr(n)}}$, define
$$\mathrm{vdim}(V):=dim_K(V/V\cap E)-dim_K(E/V\cap E)$$
 One refers to $\mathrm{vdim}(V)$ as the
 \emph{virtual dimension of $V$}. For $j\in \mathbb{Z}$ denote
$${\widehat {Gr_j(n)}}=\{V\in{\widehat {Gr(n)}}\,|\,
\mathrm{vdim}(V)=j\}$$ Then ${\widehat {Gr_j(n)}},j\in\mathbb{Z}$
give the connected components of ${\widehat {Gr(n)}}$. We have a
transitive action of $GL_n(F)$ on ${\widehat {Gr(n)}}$ with
$GL_n(A)$ as the stabilizer at the standard lattice $E$. Further,
let $\mathcal{G}_0$ be the subgroup of $GL_n(F)$, defined as,
$$\mathcal{G}_0=\{g\in GL_n(F)\,|\,\mathrm{ord(det\, }g)=0\}$$
(here, for a $f\in F$, say $f=\sum\,a_it^i$, \emph{order f} is the
smallest $r$ such that $a_r\ne 0$). Then $\mathcal{G}_0$ acts
transitively on ${\widehat {Gr_0(n)}}$ with $GL_n(A)$ as the
stabilizer at the standard lattice $E$. Also, we have a transitive
action of $SL_n(F)$ on ${\widehat {Gr_0(n)}}$ with $SL_n(A)$ as
the stabilizer at the standard lattice $E$. Thus we obtain the
identifications:
$$\begin{gathered}GL_n(F)/GL_n(A)\simeq{\widehat {Gr(n)}}\\
\mathcal{G}_0/GL_n(A)\simeq{\widehat {Gr_0(n)}},
SL_n(F)/SL_n(A)\simeq{\widehat {Gr_0(n)}}
\end{gathered}\leqno{(*)}$$ In particular, we obtain
$$\mathcal{G}_0/GL_n(A)\simeq SL_n(F)/SL_n(A)\leqno{(**)}$$

\subsection{Generators for ${\widehat W}$:}\label{gen} 
Recall that the Weyl group $\widehat W=N(K[t,t^{-1}])/T$.
Let $R$ (resp. $R^+$) be the set of roots (resp. positive roots) of $G$ relative to $B$, and let $\delta$ be the basic imaginary root of the affine Kac-Moody algebra of type $\widehat{\mathbf A}_{n-1}$ given by (cf. \cite{kac})
$$\delta=\alpha_0+\theta=\alpha_0+\cdots+\alpha_{n-1}$$
The set of real roots of $\mathcal{G}$ is given by $\left\{r\delta+\beta\mid r\in\mathbb Z, \beta\in R\right\}$, and the set of positive roots of $\mathcal{G}$ is given by $\left\{r\delta+\beta\mid r>0, \beta\in R\right\}\dot\cup R^+$ (cf. \cite{kac}).
Following the notation in \cite{kac}, we shall work with the set of generators for
${\widehat W}$ given by $\{s_0, s_1,\cdots, s_{n-1}\}$, where
$s_i, 0\le i\le n-1$ are the reflections with respect to
$\alpha_i, 0\le i\le n-1$. Note that $\{\alpha_i, 1\le i\le n-1\}$
is simply the set of simple roots of $SL_n$ (with respect to the
Borel subgroup $B$). 
In particular, the Weyl group $W$ of $SL_n(\mathbb C)$ is simply the subgroup of $\widehat W$ generated by $\{s_1,\cdots,s_{n-1}\}$. 

\def\coroots{\ensuremath Q}
\subsection{The Affine Presentation:}
\label{affine}
The generators $s_i,\, 1\leq i \leq n-1$ have the following canonical lifts to $N(K[t,t^{-1}])$:
$s_i$ is the permutation
matrix $(a_{rs})$, with $a_{jj}=1,j\ne i,i+1,\  a_{i\,i+1}=1,
a_{i+1\,i}=-1$, and all other entries are $0$. 
A canonical lift for $s_0$ is given by
$$\begin{pmatrix}
0&0&\cdots & t^{-1}\\
0&1&\cdots &0\\
\vdots & \vdots & \vdots & \vdots\\
0&\cdots &1 &0\\
-t &0&0&0
\end{pmatrix}$$
Let $s_\theta\in W$ be the reflection with respect to the longest root $\theta$ in $\mathbf A_{n-1}$ given by $\theta=\alpha_1+\cdots+\alpha_{n-1}$.
Let $L$ (resp. \coroots) be the root (resp. coroot) lattice of $\mathfrak{sl}_n(=\operatorname{Lie}(SL_n))$, and let $\langle\,,\,\rangle$ be the canonical pairing on $L\times\coroots$. 
Consider $\theta^\vee\in\coroots$ given by $\theta^\vee=\alpha_1^\vee+\cdots+\alpha_{n-1}^\vee$.
There exists (cf. \cite{kumar}, \S 13.1.6) a group isomorphism $\widehat{W}\rightarrow W\ltimes\coroots$ given by \begin{align*}
    s_i&\mapsto s_i                                 &\text{ for }1\leq i\leq n-1\\
    s_0&\mapsto s_\theta\lambda_{-\theta^{\vee}}    &
\end{align*}
where we write $\lambda_q$ for $(\operatorname{id},q)\in W\ltimes\coroots$.
In particular, we get $s_0s_\theta\mapsto\lambda_{\theta^\vee}$, which we use to compute a lift of $\lambda_{\theta^\vee}$ to $N(K[t,t^{-1}])$: \begin{align*} 
\begin{pmatrix}
0&0&\cdots & t^{-1}\\
0&1&\cdots &0\\
\vdots & \vdots & \vdots & \vdots\\
0&\cdots &1 &0\\
-t &0&0&0
\end{pmatrix}
\begin{pmatrix}
0&0&\cdots &-1\\
0&1&\cdots &0\\
\vdots & \vdots & \vdots & \vdots\\
0&\cdots &1 &0\\
1 &0&0&0
\end{pmatrix}
=\begin{pmatrix}
t^{-1}&0&\cdots &0\\
0&1&\cdots &0\\
\vdots & \vdots & \vdots & \vdots\\
0&\cdots &1 &0\\
0 &0&0&t
\end{pmatrix}
\end{align*}
Consider the element $w\in W$ corresponding to $(1,i)(i+1,n)\in S_n$, and 
observe that $w(\theta^\vee)=\alpha_i^\vee$, the $i^{th}$ simple coroot.  
It follows that a lift of $\lambda_{\alpha_i^\vee}=w\lambda_{\theta^\vee}w^{-1}$ is given by\begin{align*}
w\begin{pmatrix}
t^{-1}&0&\cdots &0\\
0&1&\cdots &0\\
\vdots & \vdots & \vdots & \vdots\\
0&\cdots &1 &0\\
0 &0&0&t
\end{pmatrix}w^{-1}
=\begin{pmatrix}
 \ddots& &&\\
 & t^{-1} & &\\
  & & t&\\
  & &  &\ddots\\
\end{pmatrix}
\end{align*}
where in the matrix on the right hand side, the dots are $1$, and the off-diagonal entries are $0$, i.e., the matrix on the right hand side is the diagonal matrix with $i,\,(i+1)$-th entries being $t^{-1},t$ respectively, and all other diagonal entries being $1$.
\\
\\
The (Coxeter) length of $\lambda_q$ is given by the following formula (cf. \cite{kumar}, \S13.1.E(3)):$$
    l(\lambda_q)=\sum\limits_{\alpha\in R^+}\lvert\alpha(q)\rvert,\qquad q\in\coroots$$ 
where $\alpha(q):=\langle\alpha,q\rangle$. 
The action of $\lambda_q$ on the root system of $\mathcal G$ is determined by the following formulae (cf. \cite{kumar}, \S13.1.6):\begin{align*}
    \lambda_q(\alpha)&=\alpha-\alpha(q)\delta,\qquad\text{ for }\alpha\in R,q\in\coroots\\
    \lambda_q(\delta)&=\delta
\end{align*} 
In particular, for $\alpha\in R^+$, $\lambda_q(\alpha)>0$ if and only if $\alpha(q)\leq 0$.
\begin{cor}
\label{count}
For $\alpha\in R^+,\,q\in\coroots$, $l(\lambda_qs_\alpha)>l(\lambda_q)$ if and only if $\alpha(q)\leq 0$.
\end{cor}
\begin{proof}
Follows from the equivalence $ws_\alpha>w$ if and only if $w(\alpha)>0$, applied to $w=\lambda_q$.
\end{proof}


\section{The element $\kappa_0$}\label{elt} 
Our goal is to give a compactification of the cotangent bundle $T^*G/B$ as a (left) $SL_n$ stable subvariety of the affine Schubert variety $X(\kappa_0)$, where $\kappa_0$ is as defined below:\begin{align*}
    \tau&:=s_{n-1}\cdots s_2s_1s_0 \\
    \kappa&:=\tau^{n-1}\\
    \kappa_0&:=w'\tau^{n-1}
\end{align*}
where $w'$ is the longest element in the Weyl group generated by $s_1,\cdots s_{n-2}$.
We first prove some properties of $\kappa$ and $\tau$ which are consequences of the braid relations
$$\begin{gathered}s_is_{i+1}s_{i}=s_{i+1}s_{i}s_{i+1},
0\le i\le n-2,\\
s_0s_{n-1}s_{0}=s_{n-1}s_{0}s_{n-1}\end{gathered}$$ 
and the commutation relations:
$$s_is_j=s_js_i, 1\le i,j\le n-1, |i-j|>1,\ \  s_0s_i=s_is_0,
2\le i\le n-2$$

\subsection{Some Facts:}\label{facts}

\ni\textbf{Fact 1:} $\tau(\delta)=\delta$ 

\ni\textbf{Fact 2:}
$\tau(\alpha_1+\cdots+\alpha_{n-1})=2\delta+\alpha_{n-1}$

\ni\textbf{Fact 3:} $\tau(r\delta+\alpha_i+\cdots+\alpha_{n-1})=
(r+1)\delta+\alpha_{i-1}+\alpha_i+\cdots+\alpha_{n-1},2\le i\le
n-1,r\in\mathbb{Z}_+$

\ni\textbf{Fact 4:} $s_{n-1}\cdots
s_{j+1}(\alpha_j)=\alpha_{j}+\alpha_{j+1}+\cdots+\alpha_{n-1},j\ne
0, n-1$

\ni\textbf{Fact 5:} $s_{n-1}\cdots
s_{1}(\alpha_0)=\delta+\alpha_{n-1}$

\ni\textbf{Fact 6:}
$\tau(\alpha_{n-1})=\delta+\alpha_{n-2}+\alpha_{n-1}$ (a special
case of Fact 3 with $r=0,i=n-1$)

\ni\textbf{Fact 7:} $\tau(\alpha_1)=\alpha_0+\alpha_{n-1}$

\ni\textbf{Fact 8:} $\tau(\alpha_i)=\alpha_{i-1},i\ne 1,n-1$

\ni\textbf{Fact 9:} $\tau(\alpha_0+\alpha_{n-1})=\alpha_{n-2}$

\begin{remark}\label{cyclic}
Facts 7, 8, 9 imply that $(\alpha_{n-1}+\alpha_0,\alpha_{n-2},\alpha_{n-3},\ldots,\alpha_1)$ is a cycle of order $n-1$ for $\tau$.
In particular, each of these roots is fixed by $\kappa$.
\end{remark}

\subsection{A reduced expression for $\kappa$}\label{kap} Let $\kappa$ be the element in
${\widehat{W}}$ defined as above. We may write $\kappa=\tau_1\cdots\tau_{n-1}$, where
$\tau_i$'s are equal, and equal to $\tau(=s_{n-1}\cdots
s_2s_1s_0)$ (we have a specific purpose behind writing $\kappa$ as
above).
\begin{lem}\label{reduced}
  The expression $\tau_1\cdots\tau_{n-1}$
for $\kappa$ is reduced.
\end{lem}
\begin{proof}

\ni\textbf{Claim:} $\tau_1\cdots\tau_{i}s_{n-1}\cdots
s_{j+1}(\alpha_j),1\le i\le n-2, 0\le j\le n-2,
\tau_1\cdots\tau_{i}(\alpha_{n-1})$, 

\ni $1\le i\le n-2$ are positive
real roots.

Note that the Claim implies the required result. We divide the
proof of the Claim into the following three cases.

\ni\textbf{Case 1:} \emph{To show:}
$\tau_1\cdots\tau_{i}(\alpha_{n-1}),1\le i\le n-2$ is a positive
real root.

\ni We have

\ni $\tau_1\cdots\tau_{i}(\alpha_{n-1})$

\ni $=\tau_1\cdots\tau_{i-1}(\delta+\alpha_{n-2}+\alpha_{n-1})$
(cf. \S \ref{facts}, Fact 6)

\ni $=\tau_1\cdots\tau_{i-2}(2\delta+\alpha_{n-3}+\alpha_{n-2}+
\alpha_{n-1})$ (cf. \S \ref{facts}, Fact 3)

\ni $=\tau_1\cdots\tau_{i-k}(k\delta+\alpha_{n-k-1}+\cdots+
\alpha_{n-1}), 0\le k\le i-1$ (cf. \S \ref{facts}, Fact 3)

\vs.2cm\ni Note that $k\le i-1$ implies that $n-k-1\ge n-i\ge 2$,
and hence we can apply \S \ref{facts}, Fact 3. Corresponding to
$k=i-1$, we obtain
$\tau_1\cdots\tau_{i}(\alpha_{n-1})=\tau_1((i-1)\delta+
\alpha_{n-i}+\cdots+\alpha_{n-1}$ ). Hence once again using \S
\ref{facts}, Fact 3, we obtain

 $$\tau_1\cdots\tau_{i}(\alpha_{n-1})=
i\delta+\alpha_{n-i-1}+\cdots+\alpha_{n-1},\ 1\le i\le n-2$$ (note
that for $1\le i\le n-2$, $n-i-1\ge 1$).

\ni\textbf{Case 2:} \emph{To show:}
$\tau_1\cdots\tau_{i}s_{n-1}\cdots s_{1}(\alpha_{0}), 1\le i\le
n-2$ is a positive real root.

\ni We have

\ni $\tau_1\cdots\tau_{i}s_{n-1}\cdots s_{1}(\alpha_{0})$

\ni $=\tau_1\cdots\tau_{i}(\delta+\alpha_{n-1})$ (cf. \S
\ref{facts}, Fact 5)

\ni $=\tau_1\cdots\tau_{i-1}(2\delta+\alpha_{n-2}+\alpha_{n-1})$
(cf. \S \ref{facts}, Fact 6)

\ni $=\tau_1\cdots\tau_{i-k}((k+1)\delta+\alpha_{n-k-1}+
\cdots+\alpha_{n-1}),0\le k\le i-1$ (cf. \S \ref{facts}, Fact 3)

\vs.2cm\ni Note that as in Case 1, for $k\le i-1$, we have,
$n-k-1\ge 2$, and therefore \S \ref{facts}, Fact 3 holds.
Corresponding to $k=i-1$, we have,

\ni $\tau_1\cdots\tau_{i}s_{n-1}\cdots
s_{1}(\alpha_{0})=\tau_1(i\delta+\alpha_{n-i}+\cdots+\alpha_{n-1})$.
 Hence once again using \S
\ref{facts} Fact 3, we obtain $$\tau_1\cdots\tau_{i}s_{n-1}\cdots
s_{1}(\alpha_{0})=(i+1)\delta+\alpha_{n-i-1}+\cdots+\alpha_{n-1},
1\le i\le n-2$$ (note that for $1\le i\le n-2$, $n-i-1\ge 1$).

\ni\textbf{Case 3:} \emph{To show:}
$\tau_1\cdots\tau_{i}s_{n-1}\cdots s_{j+1}(\alpha_j), 1\le i\le
n-2, j\ne 0,n-1$ is a positive real root.

\ni We have $\tau_1\cdots\tau_{i}s_{n-1}\cdots s_{j+1}(\alpha_j)=\tau^i(\alpha_{j}+\alpha_{j+1}+\cdots+\alpha_{n-1})$
(cf. \S \ref{facts}, Fact 4)
\ni $=\tau^i(\alpha_j)+\ldots+\tau^i(\alpha_{n-2})+\tau^i(\alpha_{n-1})$
which is positive because each term is positive (cf. Case 1 and Remark \ref{cyclic}).
\end{proof}
\begin{cor}\label{length}
$\ell(\kappa)=n(n-1)$.
\end{cor}
\subsection{Minimal representative-property for $\kappa$}\label{rep}

\begin{lem}\label{minimal} $\kappa(\alpha_i)$ is a real positive root
for all $i\ne 0$.
\end{lem}
\begin{proof} 
For $1\leq i\leq n-2$, $\kappa(\alpha_i)=\alpha_i$ is positive from Remark \ref{cyclic}.
Further, $\tau_1\cdots\tau_{n-1}(\alpha_{n-1})$

\ni $=\tau_1\cdots\tau_{n-2}(\delta+\alpha_{n-2}+\alpha_{n-1})$
(cf. \S \ref{facts}, Fact 6))

\ni $=\tau_1\cdots\tau_{n-k}((k-1)\delta+\alpha_{n-k}+\cdots+
\alpha_{n-1}), 1\le k\le n-1$ (cf. \S \ref{facts}, Fact 3))

\ni Note that for $1\le k\le n-2, n-k\ge 2$ and hence \S
\ref{facts}, Fact 3 holds. Corresponding to $k=n-1$, we get,

\ni $\tau_1\cdots\tau_{n-1}(\alpha_{n-1})$

\ni $=\tau_1((n-2)\delta+\alpha_{1}+\cdots+ \alpha_{n-1})$

\ni $=n\delta+\alpha_{n-1}$ (cf.\S \ref{facts}, Facts 1,2 )

\end{proof}

\begin{cor}\label{minimal'}
$\kappa$ is a minimal representative in ${\widehat W}/\widehat W_{G_0}$.
\end{cor}
\ni For $w\in {\widehat{W}}$, we shall denote the Schubert variety in
$\mathcal{G}/G_0$ by $X_{G_0}(w)$.
\begin{lem}\label{stable}
$X_{G_0}(\kappa)$ is stable for multiplication on the left by
$G_0$.
\end{lem}
\begin{proof}
It suffices to show that
$$s_i\kappa\le\kappa(\,mod\,\widehat{W}_{G_0}), 1\le i\le n-1\leqno{(*)}$$ 
The assertion (*) is clear if $i=n-1$. 
Observe that $ws_\alpha=s_{w(\alpha)}w$.
In particular, since $\kappa$ fixes $\alpha_i$, $1\leq i\leq n-2$, it follows $s_i\kappa=\kappa s_i=\kappa(\,mod\,\widehat W_{G_0})$, for $1\leq i\leq n-2$.
\end{proof}

\begin{lemma} 
\label{sn2}
Let \Para\ be the parabolic subgroup of $\mathcal G$ corresponding to the choice of simple roots $\left\{\alpha_1,\cdots\alpha_{n-2}\right\}$.
The element $\kappa$ is a minimal representative in $\widehat W_\Para\backslash\widehat W$.
\end{lemma}
\begin{proof}
It is enough to show that $s_i\kappa>\kappa$, or equivalently, $\kappa^{-1}(\alpha_i)>0$ for $1\leq i\leq n-2$.
This follows from Remark \ref{cyclic}.
\end{proof}

\begin{rem}
For the discussion in \S \ref{kap}, \S \ref{rep}, concerning reduced expressions, minimal-representative property and $G_0$-stability, we have used the expression for elements of $\widehat{W}$,  $\widehat{W}$ being considered as a Coxeter group. One may as well carry out the discussion using the permutation presentations  for elements of $\widehat{W}$.
\end{rem}

\begin{theorem}
[A reduced expression for $\kappa_0$] 
The element $\kappa_0(=w'\tau^{n-1})$ is the  maximal representative of $\kappa$ in $\widehat W_{G_0}\backslash\widehat W$, i.e. the unique element in $\widehat W$ such that 
$$X(\kappa_0)={\overline{G_0\kappa\mathcal{B}}}(mod\,\mathcal{B})$$
In particular, $X(\kappa_0)$ is (left) $G_0$-stable.
Let $\underline w'$ be a reduced expression for the longest element $w'$ in $\widehat W_\Para$ and $\underline\tau$ the reduced expression $s_{n-1}\cdots s_1s_0$.
Then $\underline w'\underline\tau^{n-1}$ is a reduced expression for $\kappa_0$.
\end{theorem}
\begin{proof}
Observe that $\underline w=\underline w's_{n-1}\cdots s_1$ is a reduced expression for the longest element $w$ in $\widehat W_{G_0}$, and so $w'\kappa=ws_0\tau^{n-2}$.
Lemma \ref{sn2} implies that $\underline w'\underline\tau^{n-1}$ is a reduced expression. 
In particular, $$l(\kappa_0)=l(w'\kappa)=l(w's_{n-1}\cdots s_1)+l(s_0\tau^{n-2})=l(w)+l(s_0\tau^{n-2})$$
It remains to show that $w'\kappa$ is a maximal representative in $\widehat W_{G_0}\backslash\widehat W$, i.e $s_iw'\kappa<w'\kappa$, or equivalently $l(s_iw'\kappa)<l(w'\kappa)$ for $1\leq i\leq n-1$.
First note that $$l(s_iw'\kappa)=l(s_iws_0\tau^{n-2})\leq l(s_iw)+l(s_0\tau^{n-2})$$
Now, since $w$ is the longest element in $\widehat W_{G_0}$, it follows $l(s_iw)<l(w)$ and further 
$$l(s_iw'\kappa)<l(w)+l(s_0\tau^{n-2})=l(w'\kappa)$$
\end{proof}

\section{The main lemma}\label{main} In this section, we prove one
crucial result involving $\kappa$, which we then use to prove 
the main result.
\begin{lemma}\label{crucial}
Let $Y=\sum_{1\le i<j\le n}\, a_{ij}E_{ij}$, where $E_{ij} $ is the
elementary $n\times n$ matrix with $1$ at the $(i,j)$-th place and
0's elsewhere. Let ${\underline{Y}}=Id_{n\times n}+\sum_{1\le i\le
n-1}\,t^{-i}Y^i$ (note that $Y^n=0$). Assume that $a_{i i+1}\ne 0,
1\le i\le n-1$. There exist $g\in G_0, h\in\mathcal{B}$ such that
$g\kappa = {\underline{Y}}h $
\end{lemma}
\begin{proof}
Choose $g$ to be the matrix $$g=\begin{pmatrix} 0 & 0 & 0 &\cdots
&
1\\
 -1 & 0 &0  &\cdots & g_{2n}\\
 0 &-1& 0 & \cdots & g_{3n}\\
\vdots & \vdots & \vdots & \vdots\\
0 &0& \cdots & -1 & g_{nn}
\end{pmatrix}$$
Note that the lower left corner submatrix (i.e.,the $n-1\times
n-1$ submatrix with rows $2$-nd through the $(n-1)$-th of $g$, and
the first $n-1$ columns of $g$) is $-Id_{n-1\times n-1}$, and that
determinant of $g$ equals $1$. Hence, we may take $g_{in},2\le
i\le n$ as elements in $K[[t]]$ so that $g\in G_0$. We shall now
show that there exist $g_{in}, 2\le i\le n$, and $h_{ij}, 1\le
i,j\le n $ such that $h(\in\mathcal{B})$, and $g\kappa =
{\underline{Y}}h $. We have, ${\underline{Y}}^{-1}= Id_{n\times
n}-t^{-1}Y$.  Set
$$h=(Id_{n\times n}-t^{-1}{\underline{Y}})g\kappa$$ We have
(by definition of $\kappa$ (\S\ref{elt}), and the choice of lifts
for $s_i$ (cf.\S\ref{gen}))
$$\kappa=diag(t,\cdots, t, t^{-(n-1)})$$ Note that since we want $h$ to belong to $\mathcal{B}$,
each diagonal entry in $h$ (as an element of $K[[t]]$) should have
order $0$, $h_{ij},i>j$ should have order $>0$, and $h_{ij},i<j$
should have order $\ge 0$ (since $h(0)$ should belong to $B$). Now
the diagonal entries in $h$ are
 given by $$h_{ii}=a_{ii+1}, 1\le i\le n-1, h_{nn}=
 t^{-(n-1)}g_{nn}$$
  Hence choosing $g_{nn}$ such that order$\,g_{nn}=n-1$ (note that
  since $g\in G_0$, order$\,g_{ij}\ge 0, 1\le i,j\le n$,
  so this choice for $g$ is allowed), we obtain
  that each diagonal entry in $h$ is in $K[[t]]$, with order equal to
   $0$. Also, we have
   $$\begin{gathered}h_{i+1 i}=-t, 1\le i\le n-1,\\ h_{ik}=0,k\le i-2, 3\le i\le n-1,
   h_{ik}=a_{ik+1}, 1\le i<k\le n-1
   \end{gathered}$$ Thus the entries
   $h_{ik},k\le n-1$ satisfy the order conditions mentioned above.
   Let us then consider $h_{jn}, 1\le j\le n$. We have
$$h_{jn}=t^{-(n-1)}g_{jn}-\sum_{j+1\le k\le n}\,
t^{-n}a_{jk}g_{kn}, 1\le j\le n \leqno{(*)}$$ We shall choose
$g_{in}$ (in $K[[t]]$) so that order of $g_{in}$ equals $i-1$ (
note that this agrees with the above choice of $g_{nn}$ - in the
discussion of the diagonal entries in $h$). Let us write
$$g_{in}=\sum g_{in}^{(k)}t^k$$ We shall show that with the above
 choice of $g_{in}$, the integrality condition on the $h_{in}$'s
 imposes conditions on $g_{in}^{(k)},
i-1\le k\le n, 1\le i\le n$, leading to a linear system in these
$g_{in}^{(k)}$'s
 (note that, the integrality condition on the $h_{in}$'s,
 $1\le i\le n$,  implies 
 that $h_{in}$'s should belong to $K[[t]]$, with the additional
condition that $h_{nn}$ should have order $0$ - the latter
condition having already been accommodated, since $g_{nn}$ has been
chosen to have order $n-1$). Treating $g_{in}^{(k)}$'s as the
unknowns, we show that the resulting linear system has a unique
solution, thus proving the choice of $g,h$ with the said
properties. We shall now describe this linear system. The linear
system will involve $n\choose 2$ equations in $n\choose 2$
unknowns, namely, $g_{in}^{(k)}, i-1\le k\le n, 2\le i\le n$. The
linear system is obtained as follows. The lowest power of $t$
appearing on the right hand side of (*) above is $-(n-j)$ (note
that order of $g_{kn}$ equals $k-1$). Hence equating the
coefficients of $t^{-(n-i)}, j\le i\le n-1$ on the right hand side
of (*) to $0$, we obtain
$$g_{jn}^{(i-1)}-\sum_{j+1\le k\le n}\,a_{jk}g_{kn}^{(i)}=0,  j\le
i\le n-1, 1\le j\le n-1\leqno{(**)}$$ Note that, corresponding to
$h_{nn}$, we do not have any conditions, since by our choice of
$g_{nn}$ (order of $g_{nn}$ is $n-1$), we have that $h_{nn}$ (=
$t^{-(n-1)}g_{nn}$) is integral. Also, corresponding to $g_{1n}$
(which is equal to $1$, by our choice of $g$), we have
$g_{1n}^{(i)}=0, i\ge 1$, and occurs just in one equation, namely,
the equation corresponding to the coefficient of $t^{-(n-1)}$ in
$h_{1n}$;
$$g_{1n}-a_{12}g_{2n}^{(1)}=0$$  Rewriting this equation as
$$-a_{12}g_{2n}^{(1)}=-1$$
(there is a purpose behind retaining the negative sign in
$-a_{12}g_{2n}^{(1)}$), we arrive at the linear system
$$A_nX=B$$ where $A_n$ is a square matrix of size $n\choose 2$, $X$ is
the $n\choose 2$ column matrix $(g_{jn}^k, j-1\le k\le n, 2\le
j\le n)$, and $B$ is the $n\choose 2$ column matrix with the first
entry equal to $-1$, and all other entries equal to $0$.

\ni \textbf{Claim:} $A_n$ is invertible, and $|A_n|=(-1)^{n\choose
2}\prod_{1\le i\le n-1}a_{i\,i+1}^{n-i}$.

Note that Claim implies that $(g_{jn}^{(k)}$'s, $j-1\le k\le n,
2\le j\le n)$ are uniquely determined, and therefore we may choose
$g_{jn}$ as elements in $K[[t]]$ with $(g_{jn}^{(k)}$'s, $j-1\le
k\le n$ as the solutions of the above linear system, with
$(g_{jn}^{(k)}, k>n$ being arbitrary.

We prove the Claim by induction on $n$. We shall first show that
$A_{n-1}$ can be identified in a natural way as a submatrix of
$A_n$. We want to think of the rows of $A_n$ forming $(n-1)$
blocks (referred to as \emph{row-blocks} in the sequel) of size
$n-1,n-2,\cdots,n-j,\cdots,1$, namely, the $j$-th block consists
of $n-j$ rows given by the coefficients occurring on the left hand
side of (**) for $j\ge 2$, and for $j=1$, the first block consists
of $n-1$ rows given by the coefficients occurring on the left hand
side of the following $n-1$ equations:
$$-a_{12}g_{2n}^{(1)}=-1,\ -g_{2n}^{(i)}-\sum_{3\le k\le
n}\,a_{2k}g_{kn}^{(i)}=0,2\le i\le n-1$$ Similarly, we want to
think of the columns of $A_n$ forming $(n-1)$ blocks (referred to
as \emph{column-blocks} in the sequel) of size
$n-1,n-2,\cdots,n-j,\cdots,1$, namely, the $j$-th block consisting
of $n-j$ columns indexed by $g_{jn}^{(i)}, j-1\le i\le n$. Then
indexing the $n-j$ rows in the $j$-th row-block as $j, j+1,\cdots,
n-1$, the entries in the rows of the $j$-th row-block have the
following description:

The non-zero entries in the $i$-th row in the $j$th row-block
($j\ge 2$) are

\ni $1, -a_{23},-a_{24}, \cdots, -a_{2\,i+1} $ respectively,
occurring at the columns indexed by

\ni $g_{2n}^{(i-1)},g_{3n}^{(i)},\cdots,g_{i+1\,n}^{(i)}$.

The non-zero entries in the $i$-th row in the first row-block
($j\ge 2$) are

\ni $-a_{12}, -a_{13}, \cdots, -a_{2\,i+1} $ respectively,
occurring at the columns indexed by

\ni $g_{2n}^{(i)},g_{3n}^{(i)},\cdots,g_{i+1\,n}^{(i)}$.

From this it follows that $A_{n-1}$ is obtained from $A_n$ by
deleting the first row in each row-block and the first column in
each column-block. For instance, we describe below $A_5$ and
$A_4$; for convenience of notation, we denote $b_{ij}=-a_{ij}$. We
have,
$$A_5=\left(\begin{array}{>{\columncolor{lightgray}}cccc>{\columncolor{lightgray}}ccc>{\columncolor{lightgray}}cc>{\columncolor{lightgray}}c}
\rowcolor{lightgray}
{b}_{12}&0&0&0 &0&0&0&0&0&0\\
0&b_{12}&0&0&b_{13}&0&0&0&0&0\\
0&0&b_{12}&0&0 &b_{13}&0&b_{14}&0&0\\
0&0&0&b_{12} &0&0&b_{13}&0&b_{14}&b_{15}\\
\rowcolor{lightgray}
1&0&0&0 &b_{23}&0&0&0&0&0\\
0&1&0&0&0&b_{23}&0&b_{24}&0&0\\
0&0&1&0 &0&0&b_{23}&0&b_{24}&b_{25}\\
\rowcolor{lightgray}
0&0&0& &1&0&0&b_{34}&0&0\\
0&0&0&0 &0&1&0&0&b_{34}&b_{35}\\
\rowcolor{lightgray} 0&0&0&0 &0&0&0&1&0&b_{45}
\end{array}\right)$$

$$A_4=\begin{pmatrix}

b_{12}&0&0&0&0&0\\
0&b_{12}&0 &b_{13}&0&0\\
0&0&b_{12} &0&b_{13}&b_{14}\\
1&0&0&b_{23}&0&0\\
0&1&0 &0&b_{23}&b_{24}\\
0&0&0 &1&0&b_{34}
\end{pmatrix}$$
As rows (respectively columns) of $A_5$, the positions of the
first row (respectively, the first column) in each of the four
row-blocks (respectively columns-blocks) in $A_5$ are given by
$1,5,8,10$; deleting these rows and columns in $A_5$, we get
$A_4$. These rows and columns are highlighted in $A_5$.

As above, let $b_{ij}=-a_{ij}$. Now expanding $A_n$ along the
first row, we have that $|A_n|$ equals $b_{12}|M_{1}|, M_{1}$
being the submatrix of $A_n$ obtained by deleting the first row
and first column in $A_n$ (i.e., deleting the first row
(respectively, the first column) in the first row-block
(respectively, the first column-block)). Now in $M_{1}$,  in the
first row in the second row-block the only non-zero entry is
$b_{23}$, and it is a diagonal entry in $M_{1}$. Hence expanding
$M_{1}$ through this row, we get that $|A_n|$ equals
$b_{12}b_{23}|M_{2}|, M_{2}$ being the submatrix of $A_n$ obtained
by deleting the first rows (respectively, the first columns)
in the first two row-blocks (respectively, the first two
column-blocks) in $A_n$. Now in $M_{2}$, in the first row in the
third row-block, the only non-zero entry is $b_{34}$, and it is a
diagonal entry in $M_{2}$. Hence expanding $M_{2}$ along this row,
we get that $|A_n|$ equals $b_{12}b_{23}b_{34}|M_{3}|, M_{3}$
being the submatrix of $A_n$ obtained by deleting the first
rows (respectively, the first columns) in the first three
row-blocks (respectively, the first three column-blocks) in $A_n$.
Thus proceeding, at the $(n-1)$-th step, we get that $|A_n|$
equals $b_{12}b_{23}\cdots b_{n-1\,n}|A_{n-1}|$. By induction, we
have $|A_{n-1}|=(-1)^{n-1\choose 2} \prod_{1\le i\le
n-2}a_{i\,i+1}^{n-1-i}$. Substituting back for $b_{ij}$'s, we
obtain $|A_n|=(-1)^{n\choose 2}\prod_{1\le i\le
n-1}a_{i\,i+1}^{n-i}$. It remains to verify the statement of the
claim when $n=2$ (starting point of induction). In this case, we
have $$\begin{gathered}g=\begin{pmatrix}0&1\\
-1&g_{22}\end{pmatrix}, \kappa=\begin{pmatrix}t&0\\
0&t^{-1}\end{pmatrix},\\
{\underline{Y}^{-1}}=
\begin{pmatrix}1&-t^{-1}a_{12}\\
0&1\end{pmatrix},h=\begin{pmatrix}a_{12}&t^{-1}-t^{-2}a_{12}g_{22}\\
-t&-t^{-1}g_{22}\end{pmatrix} \end{gathered}$$ Hence the linear
system consists of the single equation $$-a_{12}g_{22}^{(1)}=-1$$
Hence $A_2$ is the $1\times 1$ matrix $(-a_{12})$, and
$|A_2|=-a_{12}$, as required.
\end{proof}

\section{Lusztig's map}\label{lumap} Consider $\mathcal{N}$,
the variety of nilpotent elements in $\frak{g}$ (the Lie algebra
of $G$). In this section, we spell out (Lusztig's) isomorphism
which identifies $X_{G_0}(\kappa)$ as a compactification of $\mathcal{N}$.
\subsection{The map $\psi$:}\label{luss} Consider the map $$\psi:\mathcal{N}\rightarrow
 \mathcal{G}/G_0,
\psi(N)=(Id+t^{-1}N+t^{-2}N^2+\cdots)(mod\,G_0),N\in\mathcal{N}$$
Note that the sum on the right hand side is finite, since $N$ is
nilpotent. We now list some properties of $\psi$.

\ni\textbf{(i) $\psi$ is injective:} Let $\psi(N_1)=\psi(N_2)$.
Denoting $\lambda_i:=\psi(N_i), i=1,2$, we get that
$\lambda_2^{-1}\lambda_1$ belongs to $G_0$. On the other hand, 
$$\lambda_2^{-1}\lambda_1=(Id-t^{-1}N_2)(Id+t^{-1}N+t^{-2}N^2+\cdots)$$
Now $\lambda_2^{-1}\lambda_1$ is integral. 
It follows that both sides of the above equation equal $Id$. This
implies $\lambda_1=\lambda_2$ which in turn implies that
$N_1=N_2$. Hence we obtain the injectivity of $\psi$.

\ni\textbf{(ii) $\psi$ is $G$-equivariant:} We have

\ni $\psi(g\cdot N)=\psi(gNg^{-1})$

\ni $=(Id+t^{-1}gNg^{-1}+t^{-2}gN^2g^{-1}+\cdots)(mod\,G_0)$

\ni $=g(Id+t^{-1}N+t^{-2}N^2+\cdots)g^{-1}(mod\,G_0)$

\ni $=g(Id+t^{-1}N+t^{-2}N^2+\cdots)(mod\,G_0)$ (since $g^{-1}\in
G_0$)

\ni $=g\psi(N)$ 


\begin{prop}\label{lu}
For $N\in \mathcal{N}, \psi(N)$ belongs to $X_{G_0}(\kappa)$.
\end{prop}

\begin{proof}
We divide the proof into two cases.

\ni \textbf{Case 1:} Let $N$ be upper triangular, say,
$$N=\left(n_{ij}\right)_{1\le i,j\le n}$$ where $n_{ij}=0$,
for $i\ge j$;  note that $N\in {\underline{b}_u},
{\underline{b}_u}$ being the Lie algebra of $B_u$, the unipotent
radical of $B$. We may work in the open subset $x_{ii+1}\ne 0,1\le i\le
n-1$ in ${\underline{b}_u}$, $\sum_{1\le i<j\le n}\, x_{ij}E_{ij}$
being a generic element in ${\underline{b}_u}$. Hence we may
suppose that $n_{i\,i+1}\ne 0, 1\le i\le n-1$. In this case, in
view of Lemma \ref{crucial}, we have that there exist $g\in G_0, h\in
\mathcal{B}$ such that $g\kappa =\psi(N)h$. This implies, in view
of the $G_0$-stability for $X_{G_0}(\kappa)$ (cf. Lemma
\ref{stable}), $\psi(N)$ belongs to $X_{G_0}(\kappa)$.

\ni \textbf{Case 2:} Let $M$ be an arbitrary nilpotent matrix.
Then there exists an upper triangular matrix $N$ in the $G$-orbit
through $N$. Hence there exists a $g\in G$ such that
$M=gNg^{-1}(=g\cdot N)$ with $N$ upper triangular. Now by
$G$-equivariance of $\psi$ (cf. (ii) above), we have
$\psi(M)=g\cdot \psi(N)$. By case 1, $\psi(N)\in X_{G_0}(\kappa)$;
this together with the $G_0$-stability for $X_{G_0}(\kappa)$
implies that $\psi(M)$ belongs to $X_{G_0}(\kappa)$.
\end{proof}

\begin{theorem}\label{l-isom}
$X_{G_0}(\kappa)$ is a compactification of $\mathcal N$.
\end{theorem}
\begin{proof}
Let $\overline{\mathcal N}$ be the closure of $\mathcal N$ in $\mathcal G/G_0$.
Combining the above Proposition with \S \ref{luss}, (i) and
the facts that $\dim \mathcal{N}=n(n-1)=\dim X_{G_0}(\kappa)$ (cf. Corollaries \ref{length},\ref{minimal'}), we obtain $\overline{\mathcal{N}}=X_{G_0}(\kappa)$.
\end{proof}

\section{Cotangent bundle}
In this section, we first recall the Springer resolution.
We then construct a family $\psi_p$, parametrized by polynomials $p$ in one variable with coefficients in $\mathbb C((t))$ and constant term $1$, of maps $\psi_p:T^*G/B\rightarrow\mathcal G/\mathcal B$.
We show that for a particular choice $\phi$ in the family, we get an embedding of $T^*G/B$ inside $\mathcal G/\mathcal B$.
Using the natural projection $\mathcal G/\mathcal B\rightarrow \mathcal G/G_0$ and the results of $\S 5$, we recover the Springer resolution.
We then show that $\phi$ identifies an $SL_n$-stable closed subvariety of $X(\kappa_0)$ as a compactification of $T^*G/B$.

\label{bundle} The cotangent bundle $T^*G/B$ is a
vector bundle over $G/B$, with the fiber at any point $x\in G/B$ being
the cotangent space to $G/B$ at $x$; the dimension of $T^*G/B$
equals $2\dim G/B$. Also, $T^*G/B$ is the fiber bundle over
$G/B$ associated to the principal $B$-bundle $G\rightarrow G/B$,
for the Adjoint action of $B$ on ${\underline{b}_u}$ (the Lie
algebra of the unipotent radical $B_u$ of $B$). Thus
$$T^*G/B=G\times^B {\underline{b}_u}=G\times {\underline{b}_u}/
\sim$$ where the equivalence relation $\sim$ is given by $(g,
Y)\sim(gb, b^{-1}Yb),g\in G,Y\in{\underline{b}_u}, b\in B$ .
\subsection{Springer resolution} Let
$\mathcal{N}$ be the variety of nilpotent elements in ${\frak g}$,
the Lie algebra of $G$. Consider the map
$$\theta:G\times^B {\underline{b}_u}\rightarrow
G/B\times\mathcal{N}, \theta((g, Y))=(gB,gYg^{-1}), g\in G, Y\in
{\underline{b}_u}$$ We observe the following on the map $\theta$:

\ni \textbf{(i) $\theta$ is well defined:} Let $b\in B$. Consider
$(gb, b^{-1}Yb) (\sim (g,b))$. We have, $$\theta((gb,
b^{-1}Yb))=(gB,gb(b^{-1}Yb)b^{-1}g^{-1})=(gB,gYg^{-1})=\theta((g,
Y))$$

\ni \textbf{(ii) $\theta$ is injective:} Suppose $\theta((g_1,
Y_1))=\theta((g_2, Y_2))$. Then
$(g_1B,g_1Y_1g_1^{-1})=(g_2B,g_2Y_2g_2^{-1})$. This implies
$$g_1B=g_2B,g_1Y_1g_1^{-1}=g_2Y_2g_2^{-1}$$ Hence we obtain $$
\begin{gathered} g_1^{-1}g_2=:b\in B, Y_2=
g_2^{-1}g_1Y_1g_1^{-1}g_2\\
\therefore g_2=g_1b, Y_2=b^{-1}Y_1b\\
\therefore (g_1, Y_1)=(g_1b, b^{-1}Y_1b)=(g_2, Y_2)
\end{gathered}$$ Thus we get  an embedding $$\theta:T^*G/B\hookrightarrow G/B\times\mathcal{N}$$
The second projection $$T^*G/B\rightarrow \mathcal{N},(g,Y)\mapsto
gYg^{-1}
$$ is proper and birational and is the celebrated
\emph{Springer resolution}

\subsection{The Maps $\psi_p$}\label{maps} 
Let $p(Y)$ be a polynomial in $Y$ with coefficients in $F$, and constant term $1$.
We write \begin{align*}
    p(Y)=1+\sum\limits_{i\geq1}p_i(t)Y^i
\end{align*}
It is clear that $p(Y)\in\mathcal G$.
Define the map $\psi_p:G\times^B{\underline{b}_u}\rightarrow \mathcal{G}/\mathcal{B}$ by 
$$\psi_p(g,Y)=gp(Y)(mod\,\mathcal{B}),
g\in G, Y\in{\underline{b}_u}$$ 
The following calculation shows that $\psi_p$ is well defined: Let $g\in G,\ b\in B,\ Y\in{\underline{b}_u}$. Then \begin{align*}
\psi_p((gb, b^{-1}Yb)) &=gb\left(Id+p_1(t)b^{-1}Yb+p_2(t)b^{-1}Y^2b+\cdots\right)(mod\,\mathcal{B})\\
&=g\left(Id+p_1(t)Yb+p_2(t)Y^2b+\cdots\right)(mod\,\mathcal{B})\\
&=g\left(Id+p_1(t)Y+p_2(t)Y^2+\cdots\right)(mod\,\mathcal{B})\\
&=\psi_p(g,Y) 
\end{align*}
Also, it is clear that $\psi_p$ is $G$-equivariant. 

\subsection{Embedding of $T^*G/B$ into $\mathcal{G}/\mathcal{B}$:} 
\label{phi}
We consider one particular member $\phi$ of the family $\psi_p$: namely $\phi=\psi_p$ where $p(Y)$ is the polynomial $(1-t^{-1}Y)^{-1}$;
observe that for nilpotent $Y$, the function \begin{align*}
    p(Y)&=(1-t^{-1}Y)^{-1}\\
        &=1+t^{-1}Y+t^{-2}Y^2+\cdots
\end{align*}
is a polynomial, since the sum on the right hand side is finite.
In particular, $\phi:G\times^B
{\underline{b}_u}\rightarrow \mathcal{G}/\mathcal{B}$ is given by 
$$\phi(g,Y)=g(Id+t^{-1}Y+t^{-2}Y^2+\cdots)(mod\,\mathcal{B})$$
In the sequel, we shall denote $${\underline{Y}}:=Id+t^{-1}Y+t^{-2}Y^2+\cdots$$ We
now list some facts on the map $\phi$:

\ni\textbf{(i) $\phi$ is well-defined} 

\ni\textbf{(ii) $\phi$ is injective:} Let
$\phi((g_1,Y_1))=\phi((g_2,Y_2))$. This implies that
$g_1{\underline{Y_1}}\equiv
g_2{\underline{Y_2}}(mod\,\mathcal{B})$, where recall that for
$Y\in{\underline{b}_u},{\underline{Y}}=Id+t^{-1}Y+t^{-2}Y^2+\cdots$.
Hence, $g_1{\underline{Y_1}}= g_2{\underline{Y_2}}x$, for some
$x\in\mathcal{B}$. Denoting $h=:g_2^{-1}g_1$, we have,
$h{\underline{Y_1}}={\underline{Y_2}}x$, and
therefore,$$x={\underline{Y_2}}^{-1}h{\underline{Y_1}}=
{\underline{Y_2}}^{-1}(h{\underline{Y_1}}
h^{-1})h={\underline{Y_2}}^{-1}{\underline{Y'_1}}h$$ where
${\underline{Y'_1}}=h{\underline{Y_1}} h^{-1}$. Hence
$$xh^{-1}={\underline{Y_2}}^{-1}{\underline{Y'_1}}=(Id-t^{-1}Y_2)
(Id+t^{-1}hY_1h^{-1}+t^{-2}hY^2_1h^{-1}+\cdots)$$ 
Now, since $x\in\mathcal{B},h(=g_2^{-1}g_1)\in G$, the left hand
side is integral, i.e. it does not involve negative powers of $t$.
Hence both sides equal $Id$. This implies
$${\underline{Y_2}}={\underline{Y'_1}},x=h$$ The fact that $x=h$
together with the facts that $x\in\mathcal{B},h\in G$ implies that
$$h\in \mathcal{B}\cap G(=B)\leqno{(*)}$$ Further, the fact that
${\underline{Y_2}}={\underline{Y'_1}}$ implies that
${\underline{Y_1}}=h^{-1}{\underline{Y_2}}h$. Hence
$$Id+t^{-1}Y_1+t^{-2}Y_1^2+\cdots=Id+t^{-1}h^{-1}Y_2h+t^{-2}h^{-1}Y^2_2h
+\cdots$$ From this it follows that $$Y_1=h^{-1}Y_2h\leqno{(**)}$$
Now (*), (**) together with the fact that $h=g_2^{-1}g_1$ imply
that $$(g_1,Y_1)=(g_2h,h^{-1}Y_2h)\sim (g_2,Y_2)$$ From this
injectivity of $\phi$ follows.

\ni\textbf{(iii) $G$-equivariance:} It is clear that $\phi$ is $G$-equivariant.

\ni\textbf{(iv) Springer resolution:} Consider the projection
$\pi:\mathcal{G}/\mathcal{B}\rightarrow \mathcal{G}/G_0$. Let
$x\in T^*G/B$, say, $x=(g,Y), g\in G, Y\in{\underline{b}_u}$. We
have

\ni $\pi((g,Y))=\phi((g,Y))(mod\,G_0)$

\ni $=g(Id+t^{-1}Y+t^{-2}Y^2+\cdots)
(mod\,G_0)$

\ni $=g(Id+t^{-1}Y+t^{-2}Y^2+\cdots)g^{-1}(mod\,G_0)$

\ni $=(Id+t^{-1}N+t^{-2}N^2+\cdots) (mod\,G_0)$

\ni where $N=gYg^{-1}$ is nilpotent. Hence, in view of Lusztig's
isomorphism (cf. Proposition \ref{l-isom}), we recover the
Springer resolution as
$$\pi\,|_{T^*G/B}:T^*G/B\rightarrow \mathcal{N}\hookrightarrow
\mathcal{G}/G_0$$ 

\begin{thm}[Compactification of $T^*G/B$]\label{ctgt} 
Let $G=SL_n(\mathbb C)$ and $\phi:G/B\rightarrow\mathcal G/\mathcal B$ be as in section \ref{phi}.
Then $\phi$ identifies
${\overline{ T^*G/B}}$ (the closure being in
$\mathcal{G}/\mathcal{B}$) with a $G$-stable closed subvariety of
the affine Schubert variety $X(\kappa_0)$.
\end{thm}
\begin{proof} Let $(g_0,Y), g_0\in G, Y\in{\underline{b}_u}$. Then
$\phi(g_0,Y)=g_0(Id+t^{-1}Y+t^{-2}Y^2+\cdots)(mod\,\mathcal{B})=
g_0{\underline{Y}}(mod\,\mathcal{B})$, where
${\underline{Y}}=Id+t^{-1}Y+t^{-2}Y^2+\cdots$. Writing
$Y=\sum_{1\le i<j\le n}\, a_{ij}E_{ij}$ with $E_{ij} $ as in Lemma
\ref{crucial}, we may work in the open subset $x_{ii+1}\ne 0,1\le
i\le n-1$ in ${\underline{b}_u}$, $\sum_{1\le i<j\le n}\,
x_{ij}E_{ij}$ being a generic element in ${\underline{b}_u}$. Then
Lemma \ref{crucial} implies that there exist $g\in G_0,
h\in\mathcal{B}$ such that $g\kappa = {\underline{Y}}h $. Hence
${\underline{Y}}$ belongs to
$X(\kappa_0)(={\overline{G_0\kappa\mathcal{B}}}(mod\,\mathcal{B}))$;
hence $g_0{\underline{Y}}$ is also in $X(\kappa_0)$ (since $g_0$
is clearly in $G_0$).
\end{proof}


\section{Consequences of $\psi_p$ for $T^*G/B$}
\label{counter}
In this section, we show that for any polynomial $p$, the map $\psi_p$ as defined in \S\ref{maps} cannot realize an affine Schubert variety (in $\widehat{SL}_3/\mathcal B$) as a compactification of the cotangent bundle $T^*SL_3(K)/B$.
\begin{prop}
Let $G$ be the group $SL_3(K)$ and the $B$ the Borel subgroup of upper triangular matrices in $G$. 
Let $p$ be a polynomial as in \S\ref{maps}.
Suppose that the associated map $\psi_p:T^*G/B\rightarrow \mathcal G/\mathcal B$ is injective.
Then there exist $g\in G$, $w\in\widehat W$ and $Y\in\underline b _u$ such that $\psi_p(g,Y)\in\mathcal{B} w\mathcal{B}$ 
and $l(w)>6$. 
\end{prop}
\begin{proof}
From \S\ref{maps}, we may assume $p(Y)=1+\sum\limits_{i\geq1}p_i(t)Y^i$. 
We first claim that $p_1(t)\notin A$.
Assume the contrary.
For $$Z=\begin{pmatrix}0&1&0\\0&0&0\\0&0&0\end{pmatrix}$$
we see that $Z^2=0$, and so $p(Z)=1+p_1(t)Z\in\mathcal B$. 
In particular, $\psi_p(Z)=\psi_p(0)$, contradicting the injectivity of $\psi_p$.
\\
\\
We now write 
   $ p(Y)=1-t^{-a}qY-t^{-b}rY^2$
where \begin{itemize}
    \item $q,r\in A$. 
    \item $q(0)\neq 0$. 
    \item $a\geq 1$.
    \item Either $r=0$ or $r(0)\neq 0$.
\end{itemize}
We now fix $Y=         \begin{pmatrix}
                    0 & 1 & 0   \\
                    0 & 0 & 1   \\
                    0 & 0 & 0 
                \end{pmatrix}$ and 
$g=             \begin{pmatrix}
                    0 & 0 &-1   \\
                    0 &-1 & 0   \\
                   -1 & 0 & 0 
                \end{pmatrix}$, so that 
$$gp(Y)=\begin{pmatrix}
                    0 & 0 & -1\\
                    0 &-1 &t^{-a}q\\
                   -1 &t^{-a}q&t^{-b}r
                \end{pmatrix}$$
Our strategy is to find elements $C,D\in\mathcal B$ such that $Cgp(Y)D\in N(K[t,t^{-1}])$.
We can then identify the Bruhat cell containing $gp(Y)$, and so identify the minimal Schubert variety containing $\psi_p(g,Y)$.
The choice of $C,D$ depends on the values of certain inequalities, which we divide into $4$ cases.
We draw here a decision tree showing the relationship between the inequalities and the choice $C,D$.

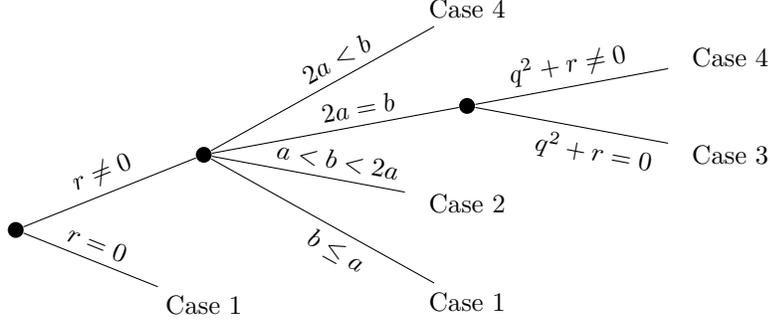
\tikzstyle{level 1}=[level distance=2.5cm, sibling distance=2cm]
\tikzstyle{level 2}=[level distance=3.5cm, sibling distance=1.3cm]

\tikzstyle{bag} = [text width=4em, text centered]
\tikzstyle{end} = [circle, minimum width=6pt,fill, inner sep=0pt]

\begin{tikzpicture}[grow=right, sloped]
\node[end] {}
    child {
        node[bag] {Case $1$}        
            edge from parent 
            node[above] {$r=0$}
    }
    child {
        node[end] {}        
        child {
                node[bag] {Case $1$}
                edge from parent
                node[below] {$\qquad b\leq a$}
            }
            child {
                node[bag] {Case $2$} 
                edge from parent
                node[above]  {$\qquad a<b<2a$}
            }
            child {
                node[end] {}
        child {
                node[bag] {Case $3$}
                edge from parent
                node[below] {$\qquad q^2+r=0$}
            }
        child {
                node[bag] {Case $4$}
                edge from parent
                node[above] {$q^2+r\neq 0$}
            }
                edge from parent
                node[above] {$\qquad 2a=b$}
            }
            child {
                node[bag] {Case $4$}
                edge from parent
                node[above] {$\qquad 2a<b$}
            }
        edge from parent         
            node[above] {$r\neq 0$}
    };
\end{tikzpicture}
\\
\\
A rational function in $t$ is implicitly equated with its Laurent power series at $0$. 
In particular, a rational function $f$ belongs to $A$ if and only if $f$ has no poles at $0$, i.e. its denominator is not divisible by $t$.
\begin{enumerate}
\item
If $r=0$ or $b\leq a$, let \begin{align*}
    C&= \begin{pmatrix}
        rt^{2a-b}+q^2   & qt^a                      & t^{2a}                        \\
        0               & -\dfrac{q}{rt^{2a-b}+q^2}  & -\dfrac{t^a}{rt^{2a-b}+q^2}    \\ 
        0               & 0                         & -\dfrac{1}{q}           
        \end{pmatrix}\\
    D&= \begin{pmatrix}
        1                           & 0             & 0                     \\
        \dfrac{qt^a}{rt^{2a-b}+q^2}  & 1             & -\dfrac{rt^{a-b}}{q}   \\  
        \dfrac{t^{2a}}{rt^{2a-b}+q^2}& 0             & 1
        \end{pmatrix} 
\end{align*}    
We compute \begin{align*}
    Cgp(Y)D=    \begin{pmatrix}
                    -t^{2a} & 0         & 0          \\
                    0       & 0         & t^{-a}    \\
                    0       & t^{-a}    & 0 
                \end{pmatrix}
\end{align*}
It follows $gp(Y)\in\mathcal B\lambda_qs_2\mathcal B$, where $q=-2a\alpha_1^\vee-a\alpha_2^\vee$.
We calculate \begin{align*}
    l(\lambda_q)&=\lvert\alpha_1(q)\rvert+\lvert\alpha_2(q)\rvert+\lvert\alpha_1(q)+\alpha_2(q)\rvert\\
                &=3a+0+3a\\
                &=6a
\end{align*}
It follows from lemma \ref{count} that $l(\lambda_qs_2)>l(\lambda_q)=6a\geq 6$.
\item
Suppose $a<b<2a$. 
In particular, $a\geq2,b\geq3$. 
Let \begin{align*}
    C&= \begin{pmatrix}
            -rt^{2a-b}+q^2  & qt^a                      & t^{2a}                            \\
            0               & -\dfrac{r}{rt^{2a-b}+q^2} & \dfrac{-qt^{b-a}}{rt^{2a-b}+q^2}  \\
            0               & 0                         & \dfrac{1}{r} 
        \end{pmatrix} \\
    D&= \begin{pmatrix}
            1                               & 0                     & 0 \\
            \dfrac{t^aq}{q^2+t^{2a-b}r}     & 1                     & 0 \\
            \dfrac{t^{2a}}{q^2+t^{2a-b}r}   & -\dfrac{qt^{b-a}}{r}  & 1 
        \end{pmatrix} \end{align*}    
We compute\begin{align*}
    Cgp(Y)R=\begin{pmatrix}
                t^{2a}  & 0         & 0         \\
                0       & t^{b-2a}  & 0         \\
                0       & 0         & t^{-b} 
                \end{pmatrix}\end{align*}
It follows $gp(Y)\in\mathcal B\lambda_q\mathcal B$, where $q=-2a\alpha_1^\vee-b\alpha_2^\vee$.
We calculate\begin{align*}
    l(\lambda_q)&=\lvert\alpha_1(q)\rvert+\lvert\alpha_2(q)\rvert+\lvert\alpha_1(q)+\alpha_2(q)\rvert\\
                &=(4a-b)+(2b-2a)+(2a+b)\\
                &=4a+2b\geq 14
\end{align*}
\item
If $b=2a$ and $q^2+r=0$, let \begin{align*}
    C= \begin{pmatrix}
            q   & t^a   & 0             \\
            0   & q     & t^a           \\
            0   & 0     & \dfrac{1}{q^2} 
        \end{pmatrix},\qquad
    D= \begin{pmatrix}
            1                   & 0             & 0 \\
            0                   & 1             & 0 \\
            -\dfrac{t^{2a}}{q^2} & \dfrac{t^a}{q} & 1 
        \end{pmatrix} \end{align*}    
We compute\begin{align*}
    Cgp(Y)D=    \begin{pmatrix}
                    0   &-t^a   & 0         \\
                    -t^a& 0     & 0         \\
                    0   & 0     &-t^{-2a} 
                \end{pmatrix}
\end{align*}
It follows $gp(Y)\in\mathcal B\lambda_qs_1\mathcal B$, where $q=-a\alpha_1^\vee-2a\alpha_2^\vee$.
Similar to the first case, we see that $l(\lambda_qs_1)>6$.
\item
Suppose either $b>2a$, or $b=2a$ and $r+q^2\neq0$. 
In particular, $r+q^2t^{b-2a}\neq0$ and $b\geq2$.
Let \begin{align*}
    C&= \begin{pmatrix}
            -r-q^2t^{b-2a}  &-qt^{b-a}                  &-t^b                               \\
            0               &-\dfrac{r}{r+q^2t^{b-2a}}  & \dfrac{qt^{b-a}}{r+q^2t^{b-2a}}   \\
            0               & 0                         & \dfrac{1}{r} 
        \end{pmatrix}\\
    D&= \begin{pmatrix}
            1                               & 0                     & 0 \\
            \dfrac{qt^{b-a}}{q^2t^{b-2a}+r} & 1                     & 0 \\
            \dfrac{t^b}{q^2t^{b-2a}+r}      & -\dfrac{qt^{b-a}}{r}  & 1 
        \end{pmatrix} 
\end{align*}    
We compute\begin{align*}
    Cgp(Y)D= \begin{pmatrix}
                    t^b & 0 & 0     \\
                    0   & 1 & 0     \\
                    0   & 0 & t^{-b} 
                \end{pmatrix}
\end{align*}
It follows $gp(Y)\in\mathcal B\lambda_q\mathcal B$, where $q=-b\alpha_1^\vee-b\alpha_2^\vee$.
We calculate\begin{align*}
    l(\lambda_q)&=\lvert\alpha_1(q)\rvert+\lvert\alpha_2(q)\rvert+\lvert\alpha_1(q)+\alpha_2(q)\rvert\\
                &=b+b+2b\\
                &=4b\geq 8
\end{align*}
\end{enumerate}
\end{proof}


\end{document}